\documentclass[12pt]{amsart}

\pdfoutput=1

\usepackage[utf8]{inputenc}
\usepackage{graphicx}
\usepackage{amsfonts}
\usepackage{amsmath}
\usepackage{amsthm}
\usepackage{amssymb}
\usepackage{pgfplots}
\pgfplotsset{compat=1.17}
\usepackage{float}
\usepackage{upgreek}
\usepackage{url}

\usepackage{siunitx}
\usepackage{booktabs}
\usepackage{enumitem}
\usepackage{multicol}
\usepackage{mathrsfs}
\usepackage{algpseudocode}
\usepackage{adjustbox}
\usepackage{algorithm}
\usepackage{mathtools}
\usepackage{tikz-cd}
\usepackage[font=small,labelfont=bf]{caption}
\usepackage{wrapfig}
\usepackage[margin=1.24in]{geometry}
\usepackage{datetime}
\usepackage[allfiguresdraft]{draftfigure}
\usepackage[
backend=biber,
style=ieee,
citestyle=numeric-comp,
sorting=nyt,
dashed=false
]{biblatex}
\usepackage{hyperref}
\usepackage{xurl}
\hypersetup{
	colorlinks=true,
	linkcolor=blue,
	citecolor=red,
	breaklinks=true
}

\newdateformat{monthyeardate}{\monthname[\THEMONTH] \THEYEAR}

\title[On accumulated spectrograms for Gabor frames]{On accumulated spectrograms\\ for Gabor frames}

\author{Simon Halvdansson}
\address{Department of Mathematical Sciences, Norwegian University of Science and Technology, 7491 Trondheim, Norway.}
\email{simon.halvdansson@ntnu.no}

\date{\monthyeardate\today}

\makeatletter
\newtheorem*{rep@theorem}{\rep@title}
\newcommand{\newreptheorem}[2]{%
	\newenvironment{rep#1}[1]{%
		\def\rep@title{#2 \ref{##1}}%
		\begin{rep@theorem}}%
		{\end{rep@theorem}}}
\makeatother

\makeatletter
\newtheorem*{rep@corollary}{\rep@title} %
\newcommand{\newrepcorollary}[2]{%
	\newenvironment{rep#1}[1]{%
		\def\rep@title{#2 \ref{##1}}%
		\begin{rep@corollary}}%
		{\end{rep@corollary}}}
\makeatother

\theoremstyle{plain}
\newtheorem{theorem}{Theorem}[section]
\newreptheorem{theorem}{Theorem}
\newtheorem*{theorem*}{Theorem}
\newtheorem{lemma}[theorem]{Lemma}
\newtheorem{proposition}[theorem]{Proposition}
\newtheorem{corollary}[theorem]{Corollary}
\newrepcorollary{corollary}{Corollary}

\theoremstyle{definition}

\theoremstyle{remark}
\newtheorem*{remark}{Remark}

\newcommand{\R}{\mathbb{R}}

\newcommand{\N}{\mathbb{N}}

\newcommand{\supp}{\operatorname{supp}}
\newcommand{\tr}{\operatorname{tr}}

\makeatletter
\newcommand{\vast}{\bBigg@{4}}
\newcommand{\Vast}{\bBigg@{5}}
\makeatother

\DeclareFontFamily{U}{mathx}{\hyphenchar\font45}
\DeclareFontShape{U}{mathx}{m}{n}{
	<5> <6> <7> <8> <9> <10>
	<10.95> <12> <14.4> <17.28> <20.74> <24.88>
	mathx10
}{}
\DeclareSymbolFont{mathx}{U}{mathx}{m}{n}
\DeclareFontSubstitution{U}{mathx}{m}{n}
\DeclareMathAccent{\widecheck}{0}{mathx}{"71}
\DeclareMathAccent{\wideparen}{0}{mathx}{"75}

\def\XXint#1#2#3{{\setbox0=\hbox{$#1{#2#3}{\int}$ }
		\vcenter{\hbox{$#2#3$ }}\kern-.6\wd0}}

\begin{document}
    \maketitle
    \begin{abstract}\vspace{-9mm}
        Analogs of classical results on accumulated spectrograms, the sum of spectrograms of eigenfunctions of localization operators, are established for Gabor multipliers. We show that the lattice $\ell^1$ distance between the accumulated spectrogram and the indicator function of the Gabor multiplier mask is bounded by the number of lattice points near the boundary of the mask and that this bound is sharp in general. The methods developed for the proofs are also used to show that the Weyl-Heisenberg ensemble restricted to a lattice is hyperuniform when the Gabor frame is tight.
        
        \vspace{3mm}
    \end{abstract}
    
    \renewcommand{\thefootnote}{\fnsymbol{footnote}}
    \renewcommand{\thefootnote}{\arabic{footnote}}
    
    \section{Introduction and main results}
    In time-frequency analysis, localization operators restrict a signal $f$ to a subset $\Omega$ of the time-frequency plane \cite{daubechies1988_loc} by means of a restricted resolution of the identity via the short-time Fourier transform (STFT) $V_g f$ as
    \begin{align*}
        A_\Omega^g f = \int_{\Omega} V_g f(z) \pi(z)g\,dz,\qquad V_gf(z) = \langle f, \pi(z) g \rangle
    \end{align*}
    where $g \in L^2(\R^d)$ is a \emph{window function}, $z = (x, \omega) \in \R^{2d}$ is a point in \emph{time-frequency space} and $\pi(z)f(t) = \pi(x, \omega) f(t) = e^{2\pi i \omega \cdot t} f(t-x)$ is a \emph{time-frequency shift} \cite{grochenig_book}. The spectral behavior of such operators has been studied extensively \cite{Feichtinger2001, DEMARI2002, Heil1994, Ramanathan1994, Feichtinger2014, Abreu2012, Marceca2023}, showing that there are approximately $|\Omega|$ eigenvalues close to $1$, followed by a \emph{plunge region} of size comparable to the length of the perimeter of $\Omega$, after which the remaining eigenvalues are close to $0$. In \cite{Abreu2015}, Abreu, Gröchenig and Romero showed that $\Omega$ can be estimated from only the spectrograms of the first $\lceil|\Omega|\rceil$ eigenfunctions using the \emph{accumulated spectrogram}
	\begin{align}\label{eq:cont_acc_spec}
            \rho_\Omega(z) = \sum_{k=1}^{\lceil |\Omega| \rceil} |V_g h_k^\Omega(z)|^2
	\end{align}
        where $(h_k^\Omega)_{k=1}^\infty$ are the eigenfunctions of the localization operator. That result was eventually improved by a sharp estimate in \cite{Abreu2017_sharp} to
	\begin{align}\label{eq:cont_acc_spec_bound}
            \Vert \rho_\Omega - \chi_\Omega \Vert_{L^1(\R^{2d})} \leq C_g |\partial \Omega|
	\end{align}
	where $\chi_\Omega$ is the indicator function of $\Omega$ and $C_g$ is a constant depending only on $g$.
	
        While these results have been numerically verified and used in the discrete setting \cite{Abreu2015, Abreu2017, Halvdansson2023_symbol, Dorfler2024, Dorfler2002}, there have been no proofs that corresponding results hold for Gabor multipliers, the discrete version of localization operators. It is the goal of this article to fill this gap by establishing versions of the main results of \cite{Abreu2015, Abreu2017_sharp} which are valid in the setting of Gabor frames. Before stating the results, we establish some notation and conventions.
	
        A \emph{Gabor frame} is a collection $\{\pi(\lambda) g\}_{\lambda \in \Lambda}$, induced by the pair $(g, \Lambda)$, where $g \in L^2(\R^d)$ is the window function and $\Lambda \subset \R^{2d}$ is a lattice, satisfying the inequalities
	\begin{align*}
            A\Vert f \Vert_{L^2}^2 \leq \sum_{\lambda \in \Lambda} |\langle f, \pi(\lambda) g \rangle|^2 \leq B\Vert f \Vert_{L^2}^2 \qquad \text{ for all }f \in L^2(\R^d)
	\end{align*}
        for a pair of \emph{frame bounds} $A, B > 0$ \cite{grochenig_book}. For a given Gabor frame there always exists a \emph{dual window} $h$ such that the reconstruction formula
	\begin{align*}
	       f = \sum_{\lambda \in \Lambda} V_g f(\lambda) \pi(\lambda) h
	\end{align*}
        holds for $f \in L^2(\R^d)$ in the $L^2$ sense. When $A=B$ the frame is said to be \emph{tight} and the dual window is a scalar multiple of $g$. The analog of localization operators in this setting, \emph{Gabor multipliers} \cite{Feichtinger2003}, are then constructed by restricting the above formula to a subset $\Omega \subset \R^{2d}$ as
	\begin{align}\label{eq:gabor_multiplier_def}
            G_{\Omega, \Lambda}^g f = \sum_{\lambda \in \Lambda} \chi_\Omega(\lambda) V_g f(\lambda) \pi(\lambda) g \qquad \text{ for }f \in L^2(\R^d).
	\end{align}
        While the tight $A=B$ situation is preferable, we will throughout the article allow for non-tight frames but still take the Gabor multipliers to be defined as in the formula above with $h = g$. This makes the operator self-adjoint and together with a compact mask $\Omega$ this guarantees that $G_{\Omega, \Lambda}^g$ is both a compact and self-adjoint operator whose eigendecomposition can be written as
	\begin{align*}
	   G_{\Omega, \Lambda}^g = \sum_{k=1}^\infty \lambda_k^\Omega (h_k^\Omega \otimes h_k^\Omega)
	\end{align*}
        where $(h_k^\Omega)_{k=1}^\infty$ is an orthonormal basis of $L^2(\R^d)$, $(h_k^\Omega \otimes h_k^\Omega)(f) = \langle f, h_k^\Omega \rangle h_k^\Omega$ is a rank-one projection operator \cite{Simon2005} and $(\lambda_k^\Omega)_{k=1}^\infty$ are the eigenvalues. The accumulated spectrogram on $\Lambda$ is then defined, analogously to \eqref{eq:cont_acc_spec}, as
        \begin{align*}
            \rho_\Omega(\lambda) = \frac{1}{\Vert g \Vert_{L^2}^2}\sum_{k=1}^{A_\Omega} |V_g h_k^\Omega(\lambda)|^2,\qquad A_\Omega = \left\lceil \frac{\#(\Omega \cap \Lambda) \Vert g \Vert_{L^2}^2}{B} \right\rceil
	\end{align*}
        since it is at most the first $A_\Omega$ eigenvalues which are close to the upper frame bound $B$ in this setting \cite{Feichtinger2024}. For technical reasons, all our main results will require the window function $g$ to belong to the space
	\begin{align*}
            M^*_\Lambda(\R^d) = \left\{ g \in L^2(\R^d) : \Vert g \Vert_{M^*_\Lambda} = \left(\sum_{\lambda \in \Lambda} |\lambda| |V_gg(\lambda)|^2\right)^{1/2} < \infty \right\}.
	\end{align*}
        By \cite[Proposition 2.1]{Feichtinger2015}, $M_\Lambda^*(\R)$ can be embedded in $L^2(\R) \cap M^{4/3}_{v_{1/2}}(\R)$ where $M^p_{v_s}(\R)$ is the weighted modulation space with weight function $v_s(z) = (1 + |z|)^s$ \cite{grochenig_book}.
    
        The boundary measure $|\partial \Omega|$ used in \eqref{eq:cont_acc_spec_bound} cannot be used in the discrete setting in general apart from when $d=1$ due to some pathological counterexamples. Details of this and how we can circumvent the problem by assuming that $\Omega$ has \emph{maximally Ahlfors regular boundary} are discussed in Section \ref{sec:bdry_measure}. To state our results generally, we need the lattice-dependent boundary measure
        \begin{align*}
            \partial^r_\Lambda \Omega = \Lambda \cap (\partial\Omega + B(0,r)).
        \end{align*}
	With it, we are ready to state our main results, the first of which is analogous to the sharp growth bound \eqref{eq:cont_acc_spec_bound}.
	\begin{theorem}\label{theorem:main_l1}
            Let $g \in M_\Lambda^*(\R^d)$ and $\Lambda$ be such that $(g, \Lambda)$ induces a frame with frame constants $A, B > 0$, $r > 0$ and $\Omega \subset \R^{2d}$ be compact. Then there exists a constant $C$ depending only on $r$ and $d$ such that
		\begin{align*}
                \Vert \rho_\Omega - \chi_\Omega \Vert_{\ell^1(\Lambda)} &\leq C\left(\Vert g \Vert_{M_\Lambda^*}^2 + 1\right) \# \partial^{r_\Lambda}_\Lambda \Omega + 2\frac{B-A}{B}\#(\Omega \cap \Lambda) + \frac{B}{\Vert g \Vert_{L^2}^2}
		\end{align*}
		where $r_\Lambda = r + l_M$ and $l_M$ is the diameter of the fundamental domain of $\Lambda$.
	\end{theorem}
        Note in particular that when the frame is tight, the error grows as $\# \partial^{r_\Lambda}_\Lambda \Omega$. This result can be used to approximate $\Omega$ directly from $\rho_\Omega$ as a level set.
        \begin{corollary}\label{cor:level}
            Let $g \in M_\Lambda^*(\R^d)$ and $\Lambda$ be such that $(g, \Lambda)$ induces a frame with frame constants $A, B > 0$, $r > 0$, $\Omega \subset \R^{2d}$ be compact and
		\begin{align*}
		      \Tilde{\Omega} = \big\{ \lambda \in \Lambda : \rho_\Omega(\lambda) > 1/2 \big\}.
		\end{align*}
		Then there exists an constant $C$ dependent only on $r$ and $d$ such that
		\begin{align*}
                \#\big( (\Omega \Delta \Tilde{\Omega}) \cap \Lambda \big) \leq C\left(\Vert g \Vert_{M_\Lambda^*}^2 + 1\right) \# \partial^{r_\Lambda}_\Lambda \Omega + 4\frac{B-A}{B} \# (\Omega \cap \Lambda) + \frac{2B}{\Vert g \Vert_{L^2}^2}
		\end{align*}
            where $\Delta$ denotes the symmetric difference of two sets, $r_\Lambda = r + l_M$ and $l_M$ is the diameter of the fundamental domain of $\Lambda$.
        \end{corollary}

	In general, it is impossible to establish a tighter bound on the $\ell^1(\Lambda)$ norm than Theorem \ref{theorem:main_l1} which we prove in Theorem \ref{theorem:sharpness} below where a special case is investigated. This result is analogous to \cite[Theorem 1.6]{Abreu2017_sharp} but in the lattice setting we need to assume some additional conditions on the window $g$.
	\begin{theorem}\label{theorem:sharpness}
            Let $g \in M_\Lambda^*(\R^d)$ and $\Lambda$ be such $(g, \Lambda)$ induces a tight frame with frame constant $1$ and 
            \begin{align*}
                |V_gg(z)| \leq C (1+|z|)^{-s},\qquad V_gg(\lambda) \neq 0\quad \text{for } \lambda \in \Lambda \cap B(0, r + 3l_M)
            \end{align*}
            for some $C > 0$ and $s > 2d+1$ where $l_M$ is the diameter of the fundamental domain of $\Lambda$. Then there exists constants $C_1, C_2$ only dependent on $g$, the lattice $\Lambda$ and the radius $r$ such that
		\begin{align*}
                C_1 \# \partial^{r_\Lambda}_\Lambda B(0, R) \leq \Vert \rho_{B(0,R)} - \chi_{B(0,R)} \Vert_{\ell^1(\Lambda)} \leq C_2 \# \partial^{r_\Lambda}_\Lambda B(0,R)
		\end{align*}
            for all $R > 0$, where $r_\Lambda = r + l_M$ and $l_M$ is the diameter of the fundamental domain of $\Lambda$.
	\end{theorem}
        The proofs of these results follow similar paths to the original results in \cite{Abreu2015, Abreu2017_sharp} and the main novelty of the present work lies in relating the eigenvalues of Gabor multipliers to our boundary measure.
  
        Our final main result is not directly related to accumulated spectrograms, but its proof uses parts of the same method used to prove the three theorems above.
 
        The Weyl-Heisenberg ensemble, originally introduced in \cite{Abreu2017} and further studied in \cite{Abreu2019, Abreu2023, Katori2021}, is a determinantal point process induced by a window function which generalizes the Ginibre ensemble. While not previously mentioned in the literature to the best of our knowledge, there is a clear discrete counterpart for Gabor frames where the point process is restricted to $\Lambda$. We are able to show that when the Gabor frame is tight, the point process is hyperuniform which an analog to one of the main results for the continuous case in \cite{Abreu2017}.
	\begin{theorem}\label{theorem:hyperuniformity}
            Let $g \in M^*_\Lambda(\R^d)$ and $\Lambda$ be such that $(g, \Lambda)$ induces a tight frame, then the determinantal point process $\mathcal{X}_g$ on $\Lambda$ with correlation kernel $K_g(\lambda, \lambda') = \langle \pi(\lambda') g, \pi(\lambda)g \rangle$ is hyperuniform.
	\end{theorem}
	In Section \ref{sec:wh_ensemble}, we give a proper definition of the point process, define hyperuniformity and prove the theorem.
 
	\subsubsection*{Notational conventions}
        The ball centered at $z \in \R^{2d}$ with radius $r$ will be denoted by $B(z, r)$. When measuring the size of a set, we will write $\#$ for cardinalities of discrete sets and $|\cdot|$ for Lebesgue measures of sets with interiors as well as the arc length of paths or the ($d-1$)-dimensional Hausdorff measure of a boundary $\partial \Omega$. For complex numbers $z$, $|z|$ will denote the absolute value as customary. The symmetric difference between two sets $A, B$ will be denoted by $A \Delta B := (A \setminus B) \cup (B \setminus A)$. The values of constants will be allowed to change between inequalities so that factors can be absorbed. For a lattice $\Lambda$, the set of summable sequences will be denoted by $\ell^1(\Lambda)$ with $\Vert c \Vert_{\ell^1(\Lambda)} = \sum_{\lambda \in \Lambda} |c(\lambda)|$ and discrete convolutions between elements of $\ell^1(\Lambda)$ will be denoted by $*_\Lambda$. 
	
	\section{Background and tools}
        In this section we collect some common tools and results which will be used throughout the article. For more background on the motivations and interpretations of accumulated spectrograms, see the original article \cite{Abreu2015}, and for more properties of Gabor multipliers and a more thorough introduction, see \cite{grochenig_book, Feichtinger2003}.  
	
	\subsection{Bounding regularization error}
        In forthcoming proofs, we will repeatedly need to bound the difference $\chi_\Omega - \chi_\Omega *_\Lambda \phi$. The main tool for the continuous version of this is \cite[Lemma 3.2]{Abreu2015}. With the goal of establishing a version of that result for the lattice setting, we prove the following lemma showing that the characteristic function $\chi_\Omega$ can be well approximated by a Schwartz function $f$.
        \begin{lemma}\label{lemma:chi_omega_f_props}
            Given a compact set $\Omega \subset \R^{2d}$ and $r > 0$, there exists a Schwartz function $f$ and a constant $C$ only dependent on the dimension $d$ such that  
		\begin{enumerate}[label=(\roman*)]
			\item $f(z) = 1$ for $z \in \Omega$,\label{item:p1}
			\item $\operatorname{supp}(f) \subset \Omega + B(0, r)$,\label{item:p2}
			\item $\Vert \nabla f \Vert_{L^\infty} \leq C/r$,\label{item:p3}
			\item $\Vert \chi_\Omega - f \Vert_{\ell^1(\Lambda)} \leq \# \partial_\Lambda^r \Omega$.\label{item:p4}
		\end{enumerate}
        \end{lemma}
        \begin{proof}
		Using the smooth bump function
            \begin{align*}
                \phi(x) = \chi_{[-1, 1]}(x) e^{\frac{-1}{1-x^2}},
            \end{align*}
            we can define the Schwarz function $\phi_r(z) = \frac{c}{r^{2d}}\phi(|z|/r)$ supported in $B(0,r)$ and by choosing $c$ appropriately we can guarantee that it integrates to $1$. Indeed, with $\omega_{d}$ the surface area of the unit sphere in $\R^d$, the integral
            \begin{align*}
                \frac{c}{r^{2d}}\int_{B(0,r)} \phi(|z|/r)\,dz &= \frac{c}{r^{2d}}\omega_{2d-1}\int_0^r x^{2d-1}\phi(x/r)\,dx\\
                &=\frac{c}{r^{2d}}\omega_{2d-1} \int_0^1 y^{2d-1}r^{2d-1}\phi(y)r\,dy\\
                &= c\omega_{2d-1} \int_0^1 y^{2d-1}\phi(y)\,dy
            \end{align*}
            is independent of $r$. We are now ready to define $f$ as
            \begin{align*}
                f(z) = \chi_{\Omega + B(0, r/2)} * \phi_{r/2}(z).
            \end{align*}
            Since $\phi_{r/2}$ is Schwarz, so is $f$ and properties \ref{item:p1} and \ref{item:p2} follow from standard properties of convolutions using that $\supp(\phi_{r/2}) \subset B(0,r/2)$. Next \ref{item:p4} follow from $\supp(\chi_\Omega - f) \subset \partial\Omega + B(0,r)$ and that $\Vert f \Vert_{L^\infty} \leq \Vert \chi_\Omega \Vert_{L^\infty} \Vert \phi_{r/2} \Vert_{L^1} = 1$ by Young's inequality. Finally for the bound on $|\nabla f|$, we can estimate
            \begin{align*}
                \Vert \nabla f \Vert_{L^\infty} &= \Vert \chi_{\Omega + B(0, r/2)} * (\nabla \phi_{r/2}) \Vert_{L^\infty}\\
                &\leq \Vert \chi_{\Omega + B(0, r/2)} \Vert_{L^\infty} \Vert \nabla \phi_{r/2} \Vert_{L^1} = \int_{B(0,r/2)} |\nabla \phi_{r/2}(z)|\,dz\\
                &=\frac{c}{(r/2)^{2d}} \int_{B(0, r/2)} \frac{1}{r/2} \left|\phi'\left( \frac{|z|}{r/2}\right)\right|\,dz\\
                &=\frac{c}{(r/2)^{2d+1}} \omega_{2d-1} \int_{0}^{r/2} x^{2d-1}\left|\phi'\left( \frac{x}{r/2}\right)\right|\,dx\\
                &=\frac{c}{(r/2)^{2d+1}} \omega_{2d-1} \frac{r^{2d}}{2^{2d}}\int_0^1 y^{2d-1}|\phi'(y)|\,dy = \frac{C}{r}
            \end{align*}
            where $C$ is a constant that collects terms independent of $r$.
            
	\end{proof}
	Using this lemma, we can establish the promised lattice version of \cite[Lemma 3.2]{Abreu2015}.
	\begin{proposition}\label{prop:approx_id_l1_est}
		Let $\phi \in \ell^1(\Lambda)$ be non-negative and satisfy
		\begin{align*}
		      1-\delta \leq \sum_{\lambda \in \Lambda} \phi(\lambda) \leq 1,\qquad \sum_{\lambda \in \Lambda} |\lambda| |\phi(\lambda)| < \infty
		\end{align*}
            for some $0 \leq \delta \leq 1$. Then there exists a constant $C$ dependent on $r > 0$ such that
		\begin{align*}
                \Vert \chi_\Omega - \chi_\Omega *_\Lambda \phi \Vert_{\ell^1(\Lambda)} \leq C\left(\sum_{\lambda \in \Lambda} |\lambda| |\phi(\lambda)| + 1\right) \# \partial^{r_\Lambda}_\Lambda \Omega + \delta \# (\Omega \cap \Lambda)
		\end{align*}
		for any compact set $\Omega \subset \R^{2d}$, where $r_\Lambda = r+l_M$ and $l_M$ is the diameter of the fundamental domain of $\Lambda$.
	\end{proposition}
	\begin{proof}
		Applying Lemma \ref{lemma:chi_omega_f_props}, we can replace $\chi_\Omega$ by $f$ using a triangle inequality argument as
		\begin{align*}
                \Vert \chi_\Omega - \chi_\Omega *_\Lambda \phi \Vert_{\ell^1(\Lambda)} \leq \underbrace{\Vert \chi_\Omega - f \Vert_{\ell^1(\Lambda)}}_{\leq \# \partial_\Lambda^r \Omega} + \Vert f - f *_\Lambda \phi \Vert_{\ell^1(\Lambda)} + \underbrace{\Vert f *_\Lambda \phi - \chi_\Omega *_\Lambda \phi \Vert_{\ell^1(\Lambda)}}_{\leq \# \partial_\Lambda^r \Omega}
		\end{align*}
            where we used Young's inequality for the estimate on the last term, Lemma \ref{lemma:chi_omega_f_props} and that $\Vert \phi \Vert_{\ell^1(\Lambda)} \leq 1$. Now it remains to show that the middle term can be bounded by a similar quantity. To do so, we will adapt \cite[Lemma 3.2]{Abreu2015} to the lattice setting. Note that
		\begin{align*}
                \Vert f &- f *_\Lambda \phi \Vert_{\ell^1(\Lambda)} = \sum_{\lambda \in \Lambda} \left| \sum_{\lambda' \in \Lambda } f(\lambda') \phi(\lambda - \lambda') - f(\lambda)\right|\\
                &= \sum_{\lambda \in \Lambda} \left| \sum_{\lambda' \in \Lambda } f(\lambda') \phi(\lambda - \lambda') - f(\lambda) \left(\sum_{\lambda' \in \Lambda}\phi(\lambda - \lambda') + 1 - \sum_{\lambda' \in \Lambda}\phi(\lambda - \lambda') \right)\right|\\
                &\leq \sum_{\lambda \in \Lambda} \left| \sum_{\lambda' \in \Lambda} \big[f(\lambda) - f(\lambda')\big] \phi(\lambda-\lambda') \right| + \sum_{\lambda \in \Lambda} \left| f(\lambda) \left( 1 - \sum_{\lambda' \in \Lambda} \phi(\lambda - \lambda') \right) \right|\\
                &\leq \sum_{\lambda \in \Lambda} \left| \sum_{\lambda' \in \Lambda} \big[f(\lambda) - f(\lambda')\big] \phi(\lambda-\lambda') \right| + \delta \Vert f \Vert_{\ell^1(\Lambda)}
		\end{align*}
		since $1 = \sum_{\lambda' \in \Lambda} \phi(\lambda') + 1 - \sum_{\lambda' \in \Lambda} \phi(\lambda')$. By elementary calculus, it holds that
		\begin{align*}
		f(\lambda) - f(\lambda') = \int_0^1 \langle \nabla f(t(\lambda - \lambda') + \lambda'), \lambda - \lambda' \rangle \,dt,
		\end{align*}
		and so we can write
		\begin{align*}
                \sum_{\lambda \in \Lambda} \left| \sum_{\lambda' \in \Lambda} \big[f(\lambda) - f(\lambda')\big] \phi(\lambda-\lambda') \right| &= \sum_{\lambda \in \Lambda} \left| \sum_{\lambda' \in \Lambda} \int_0^1 \langle \nabla f(t(\lambda - \lambda') + \lambda'), \lambda - \lambda' \rangle  \phi(\lambda-\lambda')\,dt \right|\\
                &\leq \int_0^1 \sum_{\lambda \in \Lambda} \sum_{\lambda' \in \Lambda} |\nabla f(t(\lambda - \lambda') + \lambda')| |\lambda - \lambda'| |\phi(\lambda - \lambda')|\,dt\\
                &=\int_0^1 \sum_{\lambda \in \Lambda} \sum_{\lambda' \in \Lambda} |\nabla f(t\lambda + \lambda')| |\lambda| |\phi(\lambda)|\,dt\\
                &=\sum_{\lambda \in \Lambda} |\lambda| |\phi(\lambda)| \int_0^1 \sum_{\lambda' \in \Lambda} |\nabla f(t\lambda + \lambda')|\,dt
		\end{align*}
		by Tonelli and Cauchy-Schwarz followed by a change of variables. We now claim that the final integral and sum can be uniformly bounded over all $t$ and $\lambda$. Indeed, $\nabla f$ is supported in $\partial \Omega + B(0,r)$ and $|\nabla f|$ is uniformly bounded by $C/r$ by Lemma \ref{lemma:chi_omega_f_props} so
		\begin{align*}
                \int_0^1 \sum_{\lambda' \in \Lambda} |\nabla f(t\lambda + \lambda')|\,dt \leq C \# \big( (\Lambda + \lambda t) \cap (\partial \Omega + B(0,r))\big)
		\end{align*}
            once we absorb $r$ in $C$. Now $\lambda t$ can be written as $\lambda t = \lambda_0 + z_0$ where $\lambda_0 \in \Lambda$ and $z_0$ is in the fundamental domain of $\Lambda$. This can be used to bound the quantity using the inclusion
            \begin{align*}
                (\Lambda + \lambda t) \cap (\partial \Omega + B(0,r)) &= (\Lambda + z_0) \cap (\partial \Omega + B(0,r))\\
                &= \Lambda \cap (\partial \Omega + B(z_0,r)) \subset \Lambda \cap (\partial \Omega + B(0,r_\Lambda)).
            \end{align*}
            Plugging this back into the $\Vert f - f*_\Lambda\phi \Vert_{\ell^1(\Lambda)}$ estimate yields
		\begin{align*}
                \Vert f - f *_\Lambda \phi \Vert_{\ell^1(\Lambda)} \leq C\left(\sum_{\lambda \in \Lambda} |\lambda| |\phi(\lambda)|\right) \# \partial^{r_\Lambda}_\Lambda \Omega + \delta \Vert f \Vert_{\ell^1(\Lambda)}.
		\end{align*}
            Adding back the two $\# \partial_\Lambda^r \Omega$ terms from earlier and using that $\# \partial^r_\Lambda \Omega \leq \# \partial^{r_\Lambda}_\Lambda \Omega$, we get
            \begin{align*}
                \Vert \chi_\Omega - \chi_\Omega *_\Lambda \phi \Vert_{\ell^1(\Lambda)} \leq \left( C\sum_{\lambda \in \Lambda} |\lambda| |\phi(\lambda)| + 2\right) \# \partial^{r_\Lambda}_\Lambda \Omega + \delta \Vert f \Vert_{\ell^1(\Lambda)}.
            \end{align*}
            The factor $\Vert f \Vert_{\ell^1(\Lambda)}$ can be expanded as
            \begin{align*}
                \Vert f \Vert_{\ell^1(\Lambda)} \leq \Vert f - \chi_\Omega\Vert_{\ell^1(\Lambda)} + \Vert \chi_\Omega \Vert_{\ell^1(\Lambda)} \leq \# \partial^r_\Lambda \Omega + \#(\Omega \cap \Lambda)
            \end{align*}
            and once we add this additional $\# \partial_\Lambda^r \Omega \leq \# \partial_\Lambda^{r_\Lambda} \Omega$ term and use that $\delta \leq 1$ we can absorb into $C$ to get the desired formulation.
	\end{proof}

        \subsection{Relating $\# \partial^r_\Lambda \Omega$ to $|\partial \Omega|$}\label{sec:bdry_measure}
        As mentioned in the introduction, our results are formulated with the more lattice-oriented boundary measure $\partial^r_\Lambda$ given by
        \begin{align*}
            \partial^r_\Lambda \Omega = \Lambda \cap (\partial\Omega + B(0,r))
        \end{align*}
        The following proposition clarifies the connection to the standard Lebesgue measure $|\partial \Omega|$ in the general case and shows why we cannot use it for $d > 1$ without additional assumptions.
        \begin{proposition}\label{prop:only_d1_works}
            For each radius $r$, lattice $\Lambda$ and integer $k > 0$, there exist constants $C, D > 0$ such that
            \begin{align*}
                \# \partial_\Lambda^r \Omega \leq C |\partial \Omega| + D
            \end{align*}
            for all compact $\Omega \subset \R^{2d}$ whose boundary consists of at most $k$ closed curves if and only if $d = 1$.
        \end{proposition}
        Note that the constant $C$ needs to be dependent on the number of closed curves which make up $\partial \Omega$ as one could otherwise construct a counterexample as $\Omega = \big(\Lambda + B(0,\varepsilon)\big) \cap B(0, R)$ in which case the left hand side would grow as $R^2$ and the right hand side as $R^2 \varepsilon$.
        \begin{proof}
            Let $l_m$ be the largest number so that all lattice points are separated by at least $l_m$. We first prove the inequality for $d = 1$ and then present a counterexample for $d > 1$.

            Since $\Omega$ is bounded, the set $\partial \Omega + B(0, r)$ can be covered by a finite collection $Q$ of squares with side length $l_m/\sqrt{2}$. Then each element of $\Lambda \cap (\partial \Omega + B(0, r))$ is in no more than one of these squares since the points of $\Lambda$ are spaced by at least $l_m$ and two points in a square with side length $l_m/\sqrt{2}$ are at most $l_m$ apart from each other. Formally,
		\begin{align*}
		      \# \big(\Lambda \cap (\partial \Omega + B(0,r))\big) \leq \# Q.
		\end{align*}
		Meanwhile, the squares can be encapsulated in a bigger dilation around $\partial \Omega$:
		\begin{align*}
                \bigcup_{q \in Q} q \subset \partial \Omega + B(0, r + l_m).
		\end{align*}
		As the total area of the squares is given by $\# Q \frac{l_m^2}{2}$, we can estimate
		\begin{align}\label{eq:lattice_estimate_midway}
                \# \big(\Lambda \cap (\partial \Omega + B(0, r))\big) \leq \# Q \leq \frac{2}{l_m^2} \big|\partial \Omega + B\big(0, r + l_m\big)\big|.
		\end{align}
            Now it only remains to relate this quantity to $|\partial \Omega|$. By assumption, $\partial \Omega$ can be decomposed into a collection of closed curves $\gamma_1, \dots, \gamma_k$. For each such closed curve, we claim that
		\begin{align*}
		      |\gamma + B(0,R)| \leq C |\gamma| + D
		\end{align*}
            where we have written $R$ for $r + l_m$. Indeed, if $|\gamma| = \infty$ we are done so without loss of generality, we can place a finite collection of  points $z_1, \dots, z_n$ along $\gamma$, spaced by $R$. It then holds that
		\begin{align*}
		\gamma + B(0,R) \subset \bigcup_{i =1}^n B(z_i, 2R)
		\end{align*}
            since any point in $\gamma + B(0,R)$ is within $R$ distance to a point in $\gamma$ and any point in $\gamma$ is within $R$ distance to a point $z_i$. 
		
		The number of balls, $n$, can be related to $|\gamma|$ as
		\begin{align*}
		|\gamma| \leq R\cdot n \leq |\gamma| + R
		\end{align*}
		since the distance between the centers along $\gamma$ is $R$. Putting all of this together, we can estimate
		\begin{align*}
                |\gamma + B(0, R)| &\leq \left| \bigcup_{i=1}^n B(z_i, 2R) \right|\\
                &\leq n \cdot 4\pi R^2\\
                &\leq \left( \frac{|\gamma|}{R} + 1 \right)4\pi R^2
		\end{align*}
		which proves the claim. Applying this result to \eqref{eq:lattice_estimate_midway}, we obtain
		\begin{align*}
                \#\big(\Lambda \cap (\partial \Omega + B(0, r))\big) &\leq \frac{2}{l_m^2} \sum_{i=1}^k |\gamma_k + B(0, r+l_m)|\\
                &\leq \frac{2}{l_m^2} \sum_{i=1}^k\left( \frac{|\gamma_i|}{r+l_m} + 1 \right)4 \pi \big(r+l_m\big)^2\\
                &= \frac{8\pi}{l_m^2}\left( \frac{|\partial \Omega|}{r+l_m} + k \right) \big(r + l_m\big)^2
		\end{align*}
		from which we see that both constants only depend on $\Lambda$, $r$ and $k$.
            
            We will now show that equivalence does not hold for $d > 1$ using an example where $\Omega$ is particularly elongated. Specifically, let $\Omega$ be a hyperrectangle in $2d$ dimensions with all side lengths $\varepsilon$ except for one of length $L$, i.e.,
            \begin{align*}
                \Omega = \big\{ (x_1, \dots, x_{2d}) : 0 \leq |x_1|, |x_2|,\dots, |x_{2d-1}| \leq \varepsilon/2,\, 0 \leq |x_{2d}| \leq L/2 \big\}.
            \end{align*}
            Without loss of generality by rotating $\Omega$ if necessary, we can assume that there is an infinite collection of lattice points spaced by at most $l_M$, the maximum distance between two lattice points, along the axis where $\Omega$ has a side length $L$. Consequently, we can choose $L$ large enough so that $\# \partial^r_\Lambda \Omega > D+1$. Meanwhile, the surface area $|\partial \Omega|$ can be bounded by $c L \varepsilon^{2d-2}$ so by choosing $\varepsilon$ small enough, we can make the $C |\partial \Omega|$ term arbitrarily small. In particular, if $C|\partial \Omega| < 1$ then $C|\partial \Omega| + D < D+1$ which contradicts the $\# \partial^r_\Lambda \Omega$ bound. 
        \end{proof}

        If we assume additional regularity of $\partial \Omega$, we can relate $\# \partial_\Lambda^r \Omega$ to $|\partial \Omega|$ in all dimensions. The following proposition was contributed by an anonymous referee.

        A set $\Omega$ is said to have \emph{maximally Ahlfors regular boundary} with constant $\kappa_{\partial \Omega}$ if
        \begin{align*}
            \mathcal{H}^{d-1}(\partial \Omega \cap B(z, r)) \geq \kappa_{\partial \Omega} r^{d-1}\qquad \text{for every } z\in \partial \Omega,\, 0 < r < |\partial \Omega|^{1/(d-1)}.
        \end{align*}
        This condition is not too strict and has been used for similar purposes in \cite{Marceca2023, Marceca2024}.
        \begin{proposition}
            Let $\Omega$ be a compact set with maximally Ahlfors boundary with constant $\kappa_{\partial \Omega}$. For each radius $r > 0$ and lattice $\Lambda$, there exist constants $C, D$ such that
            \begin{align*}
                \# \partial_\Lambda^r \Omega \leq C \frac{|\partial \Omega|}{\kappa_{\partial \Omega}}\left(1 + \frac{D}{|\partial \Omega|}\right).
            \end{align*}
        \end{proposition}
        \begin{proof}
            Let $F$ be a fundamental domain of $\Lambda$ with $\operatorname{diam}(F) = l_M$. Since $\Omega$ is compact it follows that the collection
            \begin{align*}
                \Lambda^* = \big\{ \lambda \in \Lambda : (\lambda + F) \cap (\partial\Omega + B(0,r)) \neq \emptyset \big\}
            \end{align*}
            is finite and the union of all translates $\lambda + F, \lambda \in \Lambda^*$ covers $\partial \Omega + B(0,r)$. This implies that
            \begin{align*}
                \big( \partial \Omega + B(0,r) \big) \subset \bigcup_{\lambda \in \Lambda^*} (\lambda + F) \subset \big( \partial \Omega + B(0, r+l_M) \big).
            \end{align*}
            Since each set $\lambda + F$ contains exactly one point in $\Lambda$, it follows that $\# \partial_\Lambda^r \Omega = \#\big( \Lambda \cap (\partial \Omega + B(0,r)) \big) \leq \# \Lambda^*$ and it remains to bound $\# \Lambda^*$. An application of \cite[Lemma 2.1]{Marceca2024} yields
            \begin{align*}
                \# \Lambda^* = \frac{1}{|F|} \left| \bigcup_{\lambda \in \Lambda^*}(\lambda + F) \right| &\leq \frac{1}{|F|}\big| \partial \Omega + B(0, r+l_m) \big|\\
                &\leq \frac{C_d}{|F|} \frac{|\partial \Omega|}{\kappa_{\partial \Omega}} (r+l_M) \left( 1+ \frac{(r+l_M)^{d-1}}{ |\partial \Omega| } \right)
            \end{align*}
            which concludes the proof.
        \end{proof}
        This result implies alternative versions of all the main results formulated with $|\partial \Omega|$ instead of $\# \partial_\Lambda^{r_\Lambda} \Omega$.

        \begin{remark}
            The hyperrectangle from Proposition \ref{prop:only_d1_works} actually has maximally Ahlfors regular boundary but the constant goes to zero as $\varepsilon \to 0$ or $L \to \infty$ which is why it worked as a counterexample.
        \end{remark}
        
	\subsection{Spectral properties of Gabor multipliers}\label{sec:prelim_spectral}
        Our key to proving all of the main theorems will be to relate them to spectral properties of Gabor multipliers. For this reason, we collect some results on the eigenvalues in this section, the first of which is the simple bound
        \begin{equation}\label{eq:eigenvalue_bounds}
            \begin{aligned}
                0 \leq \lambda_k^\Omega = \big\langle G_{\Omega, \Lambda}^g h_k^\Omega, h_k^\Omega \big\rangle &= \sum_{\lambda \in \Lambda} \chi_\Omega(\lambda) V_g h_k^\Omega(\lambda) \langle \pi(\lambda) g, h_k^\Omega \rangle\\
            &\leq \sum_{\lambda \in \Lambda} |\langle h_k^\Omega, \pi(\lambda) g \rangle|^2 \leq B\Vert h_k^\Omega \Vert_{L^2}^2 = B.
	   \end{aligned}
        \end{equation}
        In the upcoming proofs we will repeatedly have use for the fact that for any orthonormal basis $(e_n)_{n=1}^\infty$,
        \begin{align}\label{eq:easy_sum_eq_g}
            \sum_{n=1}^\infty |V_g e_n(z)|^2 = \Vert g \Vert_{L^2}^2 \qquad \text{ for all }z \in \R^{2d}.
        \end{align}
        To see that it holds, simply write out $|V_g e_n(z)|^2 = \langle e_n, \pi(z)g \rangle \langle \pi(z) g, e_n \rangle$ and sum. Note that this places an upper bound of $1$ on the accumulated spectrogram $\rho_\Omega$.
        
        The next results on the trace and Hilbert-Schmidt norm of Gabor multipliers are standard but we provide a proof in the interest of completion.
	\begin{lemma}\label{lemma:sum_of_eigs}
		The eigenvalues $\{ \lambda_k^\Omega \}_{k=1}^\infty$ of $G_{\Omega, \Lambda}^g$ satisfy
		\begin{enumerate}[label=(\roman*)]
			\item $\sum\limits_{k=1}^\infty \lambda_k^\Omega = \#(\Omega \cap \Lambda)\Vert g \Vert_{L^2}^2$,
			\item $\sum\limits_{k=1}^\infty (\lambda_k^\Omega)^2 = \sum\limits_{\lambda \in \Omega \cap \Lambda} \sum\limits_{\lambda' \in \Omega \cap \Lambda} |V_gg(\lambda-\lambda')|^2$.
		\end{enumerate}
	\end{lemma}
        \begin{proof}
            For the trace, we can compute
            \begin{align*}
                \sum_{k = 1}^\infty \big\langle G_{\Omega, \Lambda}^g h_k^\Omega, h_k^\Omega \big\rangle &= \sum_{k=1}^\infty \sum_{\lambda \in \Omega \cap \Lambda} V_g h_k^\Omega(\lambda) \langle \pi(\lambda) g, h_k^\Omega \rangle\\
                &= \sum_{\lambda \in \Omega \cap \Lambda} \sum_{k=1}^\infty |V_g h_k^\Omega(\lambda)|^2 = \sum_{\lambda \in \Omega \cap \Lambda} \Vert g \Vert_{L^2}^2 = \# (\Omega \cap \Lambda) \Vert g \Vert_{L^2}^2
            \end{align*}
            where we used \eqref{eq:easy_sum_eq_g}. Meanwhile for the sum of the squares of the eigenvalues, we can write
            \begin{align*}
                \sum_{k = 1}^\infty \big\langle G_{\Omega, \Lambda}^g h_k^\Omega, G_{\Omega, \Lambda}^g h_k^\Omega \big\rangle &= \sum_{k=1}^\infty \sum_{\lambda \in \Omega \cap \Lambda} V_g h_k^\Omega(\lambda) \langle \pi(\lambda) g, G_{\Omega, \Lambda}^g h_k^\Omega \rangle\\
                &=\sum_{k=1}^\infty \sum_{\lambda \in \Omega \cap \Lambda} V_g h_k^\Omega(\lambda)\overline{\sum_{\lambda' \in \Omega \cap \Lambda} V_g h_k^\Omega(\lambda') \langle \pi(\lambda')g, \pi(\lambda) g\rangle}\\
                &= \sum_{\lambda \in \Omega \cap \Lambda} \sum_{\lambda' \in \Omega \cap \Lambda} \sum_{k=1}^\infty \langle h_k^\Omega, \pi(\lambda) g \rangle \langle \pi(\lambda') g, h_k^\Omega \rangle \langle \pi(\lambda) g, \pi(\lambda') g \rangle\\
                &= \sum_{\lambda \in \Omega \cap \Lambda} \sum_{\lambda' \in \Omega \cap \Lambda} |V_gg(\lambda - \lambda')|^2
            \end{align*}
            since $(h_k^\Omega)_{k=1}^\infty$ is an orthonormal basis.
        \end{proof}
        The next lemma is a version of \cite[Lemma 4.14]{Feichtinger2024} which works for non-tight frames.
        \begin{lemma}\label{lemma:H_arg}
            Let $g \in L^2(\R^d)$ and $\Lambda$ be such that $(g, \Lambda)$ induces a frame with frame constants $A, B$ and $\Omega \subset \R^{2d}$ be compact. Then for each $\delta \in (0, B)$,
            \begin{align*}
                \Bigg| B\#\big\{ k : \lambda_k^\Omega &> B(1-\delta) \big\} - \#(\Omega \cap \Lambda) \Vert g \Vert_{L^2}^2\Bigg|\\
                &\leq \max\left\{ \frac{1}{\delta}, \frac{1}{1-\delta} \right\} \left| \# (\Omega \cap \Lambda) \Vert g \Vert_{L^2}^2 - \frac{1}{B}\sum_{\lambda \in \Omega \cap \Lambda} \sum_{\lambda' \in \Omega \cap \Lambda} |V_gg(\lambda - \lambda')|^2 \right|.
            \end{align*}
        \end{lemma}
        \begin{proof}
            By the eigenvalues bound \eqref{eq:eigenvalue_bounds}, the operator $H$ defined as 
            \begin{align*}
                H\big(G_{\Omega, \Lambda}^g\big) = \sum_{k=1}^\infty H(\lambda_k^\Omega) (h_k^\Omega \otimes h_k^\Omega),\qquad H(t) = \begin{cases}
                    -t \qquad\hspace{1.1mm} \text{if }0 \leq t \leq B(1-\delta),\\
                    B-t \quad \text{if } B(1-\delta) < t \leq B
                \end{cases}
            \end{align*}
            is well-defined. By applying $H$ to $G_{\Omega, \Lambda}^g$ and taking the trace we get that
            \begin{align*}
                \tr\big( H\big(G_{\Omega, \Lambda}^g\big) \big) = \sum_{k=1}^\infty H(\lambda_k^\Omega) = B\#\big\{ k : \lambda_k^\Omega > B(1-\delta) \big\} - \# (\Omega \cap \Lambda) \Vert g \Vert_{L^2}^2.
            \end{align*}
            As a function, $H$ can be bounded by
            \begin{align*}
                |H(t)| \leq \max\left\{ \frac{1}{\delta}, \frac{1}{1-\delta} \right\}\left( t - \frac{t^2}{B} \right)
            \end{align*}
            and hence
            \begin{align*}
                \Bigg|B\#\big\{ k : &\lambda_k^\Omega > B(1-\delta) \big\} - \# (\Omega \cap \Lambda) \Vert g \Vert_{L^2}^2\Bigg| = \big|\tr\big(H\big(G_{\Omega, \Lambda}^g\big)\big)\big|\leq \tr(|H|(G_{\Omega, \Lambda}^g))\\
                &\leq \max\left\{ \frac{1}{\delta}, \frac{1}{1-\delta} \right\} \left( \tr\big(G_{\Omega, \Lambda}^g\big) - \frac{1}{B}\tr\big(\big(G_{\Omega, \Lambda}^g\big)^2\big) \right)\\
                &= \max\left\{ \frac{1}{\delta}, \frac{1}{1-\delta} \right\} \left| \# (\Omega \cap \Lambda) \Vert g \Vert_{L^2}^2 - \frac{1}{B}\sum_{\lambda \in \Omega \cap \Lambda} \sum_{\lambda' \in \Omega \cap \Lambda} |V_gg(\lambda - \lambda')|^2 \right|.
            \end{align*}
        \end{proof}
	The next property is essentially a standard result \cite[Lemma 4.1]{Abreu2015} specialized to the setting of lattice convolutions $*_\Lambda$.
	\begin{lemma}\label{lemma:op_op_conv_calc}
		Let $g \in L^2(\R^d)$ and $\Omega \subset \R^{2d}$ be compact. Then
		\begin{align*}
		\sum_{k=1}^\infty \lambda_k^\Omega |V_g h_k^\Omega(\lambda)|^2 = \chi_\Omega *_{\Lambda} |V_gg|^2(\lambda).
		\end{align*}
	\end{lemma}
	\begin{proof}
		We will compute the trace $\tr\left( G_{\Omega, \Lambda}^g \pi(\lambda) (g \otimes g) \pi(\lambda)^* \right)$ using both the singular value decomposition and the definition \eqref{eq:gabor_multiplier_def} and equate the results. For the singular value decomposition, we have
		\begin{align*}
		\tr\left( \sum_{k=1}^\infty \lambda_k^\Omega (h_k^\Omega \otimes h_k^\Omega) \pi(\lambda) (g \otimes g) \pi(\lambda)^*  \right) &= \sum_{k=1}^\infty \lambda_k^\Omega \tr\Big( (h_k^\Omega \otimes h_k^\Omega) \pi(\lambda) (g \otimes g) \pi(\lambda)^* \Big)\\
		&=\sum_{k=1}^\infty \lambda_k^\Omega \sum_{n=1}^\infty \big\langle (h_k^\Omega \otimes h_k^\Omega) \pi(\lambda) (g \otimes g) \pi(\lambda)^* e_n, e_n \big\rangle\\
		&= \sum_{k=1}^\infty \lambda_k^\Omega \sum_{n=1}^\infty \big\langle e_n, \pi(\lambda) g\big\rangle \big\langle \pi(\lambda) g, h_k^\Omega \big\rangle \big\langle h_k^\Omega, e_n \big\rangle\\
		&=\sum_{k=1}^\infty \lambda_k^\Omega |V_g h_k^\Omega(\lambda)|^2.
		\end{align*}
		Meanwhile the trace can also be computed as
		\begin{align*}
		\tr\left( G_{\Omega, \Lambda}^g \pi(\lambda) (g \otimes g) \pi(\lambda)^* \right) &= \sum_{n=1}^\infty \big\langle G_{\Omega, \Lambda}^g \pi(\lambda) (g \otimes g) \pi(\lambda)^* e_n, e_n \big\rangle\\
		&=\sum_{n=1}^\infty \big\langle \pi(\lambda)^* e_n, g\big\rangle \big\langle G_{\Omega, \Lambda}^g \pi(\lambda) g, e_n \big\rangle\\
		&= \sum_{n=1}^\infty \big\langle e_n, \pi(\lambda) g \big\rangle \left\langle \sum_{\lambda' \in \Lambda} \chi_\Omega(\lambda') \langle \pi(\lambda)g, \pi(\lambda')g \right\rangle \big\langle \pi(\lambda') g, e_n\big\rangle\\
		&= \sum_{\lambda' \in \Lambda} \chi_\Omega(\lambda') |\langle \pi(\lambda') g, \pi(\lambda) g\rangle |^2\\
		&= \chi_\Omega *_\Lambda |V_gg|^2(\lambda)
		\end{align*}
		which finishes the proof.
	\end{proof}
	We this result, we are ready to apply Proposition \ref{prop:approx_id_l1_est} on spectral quantities.
        \begin{lemma}\label{lemma:sum_to_boundary_estimate}
            Let $g \in M^*_\Lambda(\R^d)$ and $\Lambda$ be such that $(g, \Lambda)$ induces a frame with frame constants $A, B > 0$, and $\Omega \subset \R^{2d}$ be compact. Then
		\begin{align*}
                \Bigg| \frac{1}{B} \sum_{\lambda \in \Lambda} \sum_{\lambda' \in \Lambda} &\chi_\Omega(\lambda) \chi_\Omega(\lambda') |V_gg(\lambda - \lambda')|^2 - \# (\Omega \cap \Lambda)\Vert g \Vert_{L^2}^2  \Bigg|\\
                &\leq C\Vert g \Vert_{L^2}^2 \left(\Vert g \Vert_{M_\Lambda^*}^2 + 1 \right) \# \partial^{r_\Lambda}_\Lambda \Omega + \frac{B-A}{B} \# (\Omega \cap \Lambda)\Vert g \Vert_{L^2}^2
		\end{align*}
		for a constant $C$ depending only on $r$ and $d$.
	\end{lemma}
	\begin{proof}
		We estimate
		\begin{align*}
                \Bigg| \frac{1}{B}\sum_{\lambda \in \Lambda} &\sum_{\lambda' \in \Lambda} \chi_\Omega(\lambda) \chi_\Omega(\lambda') |V_gg(\lambda - \lambda')|^2 - \# (\Omega \cap \Lambda)\Vert g \Vert_{L^2}^2  \Bigg|\\
                &= \Bigg| \frac{1}{B}\sum_{\lambda \in \Lambda} \chi_\Omega(\lambda) \big(\chi_\Omega *_\Lambda |V_gg|^2(\lambda)\big) - \Vert g \Vert_{L^2}^2\sum_{\lambda \in \Lambda} \chi_\Omega(\lambda)\Bigg| \\
                &= \Vert g \Vert_{L^2}^2\Bigg| \sum_{\lambda \in \Lambda} \chi_\Omega(\lambda) \Bigg(\frac{1}{B\Vert g \Vert_{L^2}^2}\chi_\Omega *_\Lambda |V_gg|^2(\lambda) - \chi_\Omega(\lambda)\Bigg)\Bigg|\\
                &\leq \Vert g \Vert_{L^2}^2\sum_{\lambda \in \Lambda} \Bigg| \frac{1}{B \Vert g \Vert_{L^2}^2}\chi_\Omega *_\Lambda |V_gg|^2(\lambda) - \chi_\Omega(\lambda)\Bigg|\\
                &= \Vert g \Vert_{L^2}^2\Bigg\Vert \chi_\Omega *_\Lambda \frac{|V_gg|^2}{B \Vert g \Vert_{L^2}^2} - \chi_\Omega \Bigg\Vert_{\ell^1(\Lambda)}.
		\end{align*}
            Now since $A \Vert g \Vert_{L^2}^2 \leq \sum_{\lambda \in \Lambda} |V_gg(\lambda)^2| \leq B \Vert g \Vert_{L^2}^2$ by the frame inequality, we can apply Proposition \ref{prop:approx_id_l1_est} with $\phi = \frac{|V_gg|^2}{B \Vert g \Vert_{L^2}^2}$ and $\frac{A}{B} = 1-\delta$ to the final norm to obtain the desired bound.
	\end{proof}        
	
	\subsection{Lattice Gabor space}\label{sec:prelim_rkhs}
        The image space of the standard short-time Fourier transform, the so called \emph{Gabor space} $V_g(L^2) \subset L^2(\R^{2d})$, is a reproducing kernel Hilbert space (RKHS) with reproducing kernel $K_g(z, w) = \langle \pi(w)g, \pi(z)g \rangle$ \cite{Feichtinger2014}. The Toeplitz operators on this space, Gabor-Toeplitz operators, are unitarily equivalent to localization operators via conjugation with the STFT. Similarly, it can be shown that the image of $L^2(\R^d)$ under our STFT, which maps to $\ell^2(\Lambda)$, also is a reproducing kernel Hilbert space \cite[Section 5]{Feichtinger2014} with reproducing kernel
	\begin{align}\label{eq:lattice_gabor_rk}
            K_g(\lambda, \lambda') = \langle \pi(\lambda')g, \pi(\lambda) g \rangle.
	\end{align}
	\section{Accumulated spectrograms}    
	In this section we prove all the theorems on accumulated spectrograms. As we will see, the proofs generally follow those from \cite{Abreu2015} and \cite{Abreu2017_sharp}.
	
	\subsection{$\ell^1$ estimate}
	We first set out to prove our most important result, Theorem \ref{theorem:main_l1}, using the approach from \cite{Abreu2017_sharp}.
	
        The following lemma allows us to move from $\Vert \rho_\Omega - \chi_\Omega \Vert_{\ell^1(\Lambda)}$ to a purely spectral quantity since $\tr(G_{\Omega, \Lambda}^g) = \# (\Omega \cap \Lambda)\Vert g \Vert_{L^2}^2$ by Lemma \ref{lemma:sum_of_eigs}.
	
	\begin{lemma}\label{lemma:l1_eigenval_rewrite}
		Let $g \in L^2(\R^d)$ and $\Lambda$ be such that $(g, \Lambda)$  induces a frame with frame constants $A, B > 0$, and $\Omega \subset \R^{2d}$ be compact. Then 
		\begin{align*}
                \Vert \rho_\Omega - \chi_\Omega \Vert_{\ell^1(\Lambda)} \leq \frac{2}{\Vert g \Vert_{L^2}^2} \left( \# (\Omega \cap \Lambda)\Vert g\Vert_{L^2}^2 - \sum_{k=1}^{A_\Omega} \lambda_k^\Omega \right) + \frac{B}{\Vert g \Vert_{L^2}^2}.
		\end{align*}
	\end{lemma}
	\begin{proof}
            We first rewrite the difference $\Vert \rho_\Omega - \chi_\Omega \Vert_{\ell^1(\Lambda)}$ as
            \begin{align}\label{eq:l1_diff_unnormalized}
                \Vert \rho_\Omega - \chi_\Omega \Vert_{\ell^1(\Lambda)} = \frac{1}{\Vert g \Vert_{L^2}^2} \left\Vert \sum_{k=1}^{A_\Omega} |V_g h_k^\Omega|^2 - \Vert g \Vert_{L^2}^2 \chi_\Omega \right\Vert_{\ell^1(\Lambda)}.
            \end{align}
		Now note that the eigenvalues of $G_{\Omega, \Lambda}^g$ can be written as
		\begin{align}\label{eq:gabor_eigenvalue}
                \lambda_k^\Omega = \big\langle G_{\Omega, \Lambda}^g h_k^\Omega, h_k^\Omega \big\rangle = \sum_{\lambda \in \Omega \cap \Lambda} V_g h_k^\Omega(\lambda) \langle \pi(\lambda) g, h_k^\Omega \rangle = \sum_{\lambda \in \Omega \cap \Lambda} |V_gh_k^\Omega(\lambda)|^2.
		\end{align}
            The $\ell^1(\Lambda)$ difference in \eqref{eq:l1_diff_unnormalized} can be split into two parts, the interior and exterior of $\Omega$. For the interior, we have that $\chi_\Omega(\lambda) = 1$ and $\sum_{k=1}^{A_\Omega} |V_g h_k^\Omega(\lambda)|^2 \leq \Vert g \Vert_{L^2}^2$ by \eqref{eq:easy_sum_eq_g}, so
		\begin{align*}
                \sum_{\lambda \in \Omega \cap \Lambda} \left|\sum_{k=1}^{A_\Omega} |V_g h_k^\Omega(\lambda)|^2 - \Vert g \Vert_{L^2}^2 \chi_\Omega(\lambda)\right| &=\# (\Omega \cap \Lambda)\Vert g \Vert_{L^2}^2 - \sum_{k=1}^{A_\Omega}\sum_{\lambda \in \Omega \cap \Lambda} |V_g h_k^\Omega(\lambda)|^2\\
                &=\# (\Omega \cap \Lambda)\Vert g \Vert_{L^2}^2 - \sum_{k=1}^{A_\Omega}\lambda_k^\Omega
		\end{align*}
		where we used \eqref{eq:gabor_eigenvalue} for the second step. Meanwhile for the exterior where $\chi_\Omega(\lambda) = 0$,
		\begin{align}\nonumber
                \sum_{\lambda \in \Omega^c \cap \Lambda} \left| \sum_{k=1}^{A_\Omega} |V_g h_k^\Omega(\lambda)|^2 - \Vert g \Vert_{L^2}^2 \chi_\Omega(\lambda)\right| &= \sum_{k=1}^{A_\Omega}\sum_{\lambda \in \Lambda} |V_g h_k^\Omega(\lambda)|^2 - \sum_{k=1}^{A_\Omega}\sum_{\lambda \in \Omega \cap \Lambda} |V_g h_k^\Omega(\lambda)|^2\\\label{eq:non_tight_ineq}
                &\leq B A_\Omega - \sum_{k=1}^{A_\Omega} \lambda_k^\Omega\\\nonumber
                &\leq B + \# (\Omega \cap \Lambda) \Vert g \Vert_{L^2}^2 - \sum_{k=1}^{A_\Omega} \lambda_k^\Omega
		\end{align}
            where we used that the upper frame bound is $B$. Combining these two results, we get the expression in the statement of the lemma.
	\end{proof}
        Note that the only inequalities in the above proof stem from the frame being non-tight and $\frac{\# (\Omega \cap \Lambda)\Vert g \Vert_{L^2}^2}{B}$ not being an integer.
	\begin{reptheorem}{theorem:main_l1}
            Let $g \in M_\Lambda^*(\R^d)$ and $\Lambda$ be such that $(g, \Lambda)$ induces a frame with frame constants $A, B > 0$, $r > 0$ and $\Omega \subset \R^{2d}$ be compact. Then there exists a constant $C$ depending only on $r$ and $d$ such that
		\begin{align*}
                \Vert \rho_\Omega - \chi_\Omega \Vert_{\ell^1(\Lambda)} &\leq C\left(\Vert g \Vert_{M_\Lambda^*}^2 + 1\right) \# \partial^{r_\Lambda}_\Lambda \Omega + 2\frac{B-A}{B}\#(\Omega \cap \Lambda) + \frac{B}{\Vert g \Vert_{L^2}^2}
		\end{align*}
		where $r_\Lambda = r + l_M$ and $l_M$ is the diameter of the fundamental domain of $\Lambda$.
	\end{reptheorem}
	\begin{proof}
		By Lemma \ref{lemma:sum_of_eigs} followed by Lemma \ref{lemma:sum_to_boundary_estimate}, it holds that
		\begin{equation}\label{eq:trace_difference}
                \begin{aligned}
                    0 &\leq \tr\big(G_{\Omega, \Lambda}^g\big) - \frac{1}{B}\tr\big( (G_{\Omega, \Lambda}^g)^2 \big)\\
                    &\leq C\Vert g \Vert_{L^2}^2 \left(\Vert g \Vert_{M_\Lambda^*}^2 + 1\right) \# \partial^{r_\Lambda}_\Lambda \Omega + \frac{B-A}{B}\# (\Omega \cap \Lambda)\Vert g \Vert_{L^2}^2.
                \end{aligned}
		\end{equation}
		We can also write the trace difference \eqref{eq:trace_difference}, with $A_\Omega = \big\lceil\frac{\# (\Omega \cap \Lambda)\Vert g \Vert_{L^2}^2}{B}\big\rceil$, as
		\begin{align*}
                \tr\big(G_{\Omega, \Lambda}^g\big) &- \frac{1}{B}\tr\big( (G_{\Omega, \Lambda}^g)^2 \big) = \frac{1}{B}\sum_{k=1}^\infty \lambda_k^\Omega(B - \lambda_k^\Omega)\\
                &= \frac{1}{B}\sum_{k=1}^{A_\Omega} \lambda_k^\Omega (B - \lambda_k^\Omega) + \frac{1}{B}\sum_{k= A_\Omega+1}^\infty \lambda_k^\Omega (B - \lambda_k^\Omega)\\
                &\geq \frac{\lambda_{A_\Omega}^\Omega}{B} \sum_{k=1}^{A_\Omega} (B-\lambda_k^\Omega) + \big(B-\lambda_{A_\Omega}^\Omega\big)\frac{1}{B} \sum_{k = A_\Omega+1}^\infty \lambda_k^\Omega\\
                &= \lambda_{A_\Omega}^\Omega A_\Omega - \frac{\lambda_{A_\Omega}^\Omega}{B} \sum_{k=1}^{A_\Omega} \lambda_k^\Omega + \big(B-\lambda_{A_\Omega}^\Omega \big) \frac{1}{B} \left( \# (\Omega \cap \Lambda)\Vert g \Vert_{L^2}^2 - \sum_{k=1}^{A_\Omega} \lambda_k^\Omega \right)\\
                &= \lambda_{A_\Omega}^\Omega A_\Omega + \# (\Omega \cap \Lambda)\Vert g \Vert_{L^2}^2 \left(1-\frac{\lambda_{A_\Omega}^\Omega}{B} \right) - \sum_{k=1}^{A_\Omega} \lambda_k^\Omega\\
                &= \# (\Omega \cap \Lambda)\Vert g \Vert_{L^2}^2 - \sum_{k=1}^{A_\Omega} \lambda_k^\Omega + \lambda_{A_\Omega}^\Omega \left(A_\Omega - \frac{\# (\Omega \cap \Lambda)\Vert g \Vert_{L^2}^2}{B}\right)\\
                &\geq \# (\Omega \cap \Lambda)\Vert g \Vert_{L^2}^2 - \sum_{k=1}^{A_\Omega} \lambda_k^\Omega.
		\end{align*}
		Combining the above with \eqref{eq:trace_difference} yields
		\begin{align*}
                \# (\Omega \cap \Lambda)\Vert g \Vert_{L^2}^2 - \sum_{k=1}^{A_\Omega} \lambda_k^\Omega &\leq C\Vert g \Vert_{L^2}^2 \left(\Vert g \Vert_{M_\Lambda^*}^2+ 1\right) \# \partial^{r_\Lambda}_\Lambda \Omega + \frac{B-A}{B}\# (\Omega \cap \Lambda)\Vert g \Vert_{L^2}^2.
		\end{align*}
		We can use this estimate together with Lemma \ref{lemma:l1_eigenval_rewrite} to conclude that
		\begin{align*}
                \big\Vert \rho_\Omega - \chi_\Omega \big\Vert_{\ell^1(\Lambda)} &\leq \frac{2}{\Vert g \Vert_{L^2}^2}\left( \# (\Omega \cap \Lambda)\Vert g \Vert_{L^2}^2 - \sum_{k=1}^{A_\Omega} \lambda_k^\Omega \right) + \frac{B}{\Vert g \Vert_{L^2}^2}\\
                &\leq 2C\left(\Vert g \Vert_{M_\Lambda^*}^2 + 1\right) \# \partial^{r_\Lambda}_\Lambda \Omega + 2\frac{B-A}{B}\# (\Omega \cap \Lambda) + \frac{B}{\Vert g \Vert_{L^2}^2}
		\end{align*}
		which, after absorbing the factor 2, is what we wished to show.
	\end{proof}
	
	\begin{remark}
            The above proof can be extended to apply to multi-window Gabor multipliers \cite{Dorfler2009, Luef2019} or, more generally, mixed-state Gabor multipliers \cite{Skrettingland2020, Luef2019, Feichtinger2024} in the same way as was done in \cite{Luef2019_acc} as Lemma \ref{lemma:sum_of_eigs} and Lemma \ref{lemma:sum_to_boundary_estimate} both are valid in this setting. In the notation of \cite{Feichtinger2024}, this means that if $S$ is a trace-class operator such that $(S, \Lambda)$ is a mixed-state Gabor frame and $G_{\Omega, \Lambda}^S = \sum_{k=1}^\infty \lambda_k^\Omega (h_k^\Omega \otimes h_k^\Omega)$, we can set $\rho_\Omega^S = \frac{1}{\tr(S)}\sum_{k=1}^{A_\Omega} Q_S(h_k^\Omega)$ with $A_\Omega = \big\lceil \frac{\# (\Omega \cap \Lambda) \tr(S)}{B} \big\rceil$ and it will hold that
		\begin{align*}
                \Vert \rho_\Omega^S - \chi_\Omega \Vert_{\ell^1(\Lambda)} &\leq C\left(\sum_{\lambda \in \Lambda} |\lambda| \tilde{S}(\lambda) + 1 \right) \# \partial^{r_\Lambda}_\Lambda \Omega + 2\frac{B-A}{B}\# (\Omega \cap \Lambda) + \frac{B}{\tr(S)}
		\end{align*}
            for the same constant $C$. This type of generalization is analogous to that for the accumulated Cohen's class in \cite{Luef2019_acc}.
	\end{remark}
	
	\begin{remark}
            Parallel to the development of accumulated spectrograms, there have been corresponding results for the eigenvalues and eigenfunctions of \emph{wavelet} localization operators, see \cite{Wong2002, Abreu2012, Halvdansson2023, Ghobber2017, DEMARI2002}. It is likely that results similar to Theorem \ref{theorem:main_l1} hold in the frame setting for accumulated scalograms but we make no attempts to prove this here.
	\end{remark}

        Theorem \ref{theorem:main_l1} can be used to approximate $\Omega$ as the set where $\rho_\Omega > 1/2$. The task of approximately inverting the mapping $\Omega \mapsto G_{\Omega, \Lambda}^g$ has been studied elsewhere in the continuous setting \cite{Abreu2012, Abreu2015, Halvdansson2023_symbol, Romero2024}, but not in the discrete case.
        \begin{repcorollary}{cor:level}
            Let $g \in M_\Lambda^*(\R^d)$ and $\Lambda$ be such that $(g, \Lambda)$ induces a frame with frame constants $A, B > 0$, $r > 0$, $\Omega \subset \R^{2d}$ be compact and
		\begin{align*}
		      \Tilde{\Omega} = \big\{ \lambda \in \Lambda : \rho_\Omega(\lambda) > 1/2 \big\}.
		\end{align*}
		Then there exists a constant $C$ dependent only on $r$ and $d$ such that
		\begin{align*}
                \#\big( (\Omega \Delta \Tilde{\Omega}) \cap \Lambda \big) \leq C\left(\Vert g \Vert_{M_\Lambda^*}^2 + 1\right) \# \partial^{r_\Lambda}_\Lambda \Omega + 4\frac{B-A}{B} \# (\Omega \cap \Lambda) + \frac{2B}{\Vert g \Vert_{L^2}^2}
		\end{align*}
            where $\Delta$ denotes the symmetric difference of two sets, $r_\Lambda = r + l_M$ and $l_M$ is the diameter of the fundamental domain of $\Lambda$.
	\end{repcorollary}
        \begin{proof}
            Define $E = \{ \lambda \in \Lambda : |\rho_\Omega(\lambda) - \chi_\Omega(\lambda)| \geq 1/2\}$, then we can bound the cardinality of $E$ using Chebyshev's inequality as
            \begin{align*}
                \# E &= \# \big\{ \lambda \in \Lambda : |\rho_\Omega(\lambda) - \chi_\Omega(\lambda)| \geq 1/2 \big\} \leq 2\Vert \rho_\Omega - \chi_\Omega \Vert_{\ell^1(\Lambda)}\\
                &\leq 2C\left(\Vert g \Vert_{M_\Lambda^*}^2 + 1\right) \# \partial^{r_\Lambda}_\Lambda \Omega + 4\frac{B-A}{B} \# (\Omega \cap \Lambda) + \frac{2B}{\Vert g \Vert_{L^2}^2}
            \end{align*}
            using Theorem \ref{theorem:main_l1}. We claim that $(\Omega \Delta \Tilde{\Omega}) \cap \Lambda \subset E$. Indeed, if $\lambda \in \Lambda$ is in $\Omega$ but not $\tilde{\Omega}$, then $\chi_\Omega(\lambda) = 1$ and $\rho_\Omega(\lambda) \leq 1/2$ so $|\rho_\Omega(\lambda) - \chi_\Omega(\lambda)| \geq 1/2$. Meanwhile if $\lambda$ is in $\tilde{\Omega}$ but not $\Omega$, then $\chi_\Omega(\lambda) = 0$ and $\rho_\Omega(\lambda) > 1/2$ so $|\rho_\Omega(\lambda) - \chi_\Omega(\lambda)| > 1/2$. This proves the inclusion from which the result follows.
	\end{proof}

	\subsection{Sharpness of estimate}
        On top of improving the $L^1$ estimate of $\rho_\Omega - \chi_\Omega$, \cite{Abreu2017_sharp} also showed that it is impossible to establish stronger bounds on the $L^1$ distance by bounding the growth from below in the special case of dilated balls. In this section, we do the same by following a similar strategy as that used to prove \cite[Theorem 1.6]{Abreu2017_sharp} specialized to the lattice case. From now on we will assume that the Gabor frame is tight with frame constant $1$ as we can only guarantee that $\Vert \rho_{B(0,R)} - \chi_{B(0,R)} \Vert_{\ell^1(\Lambda)}$ grows slower than the area of $\Omega$ when this is the case. Note that the frame constant being $1$ is not a restriction as one can multiply $g$ by a constant if this is not the case.
	\begin{reptheorem}{theorem:sharpness}
            Let $g \in M_\Lambda^*(\R^d)$ and $\Lambda$ be such $(g, \Lambda)$ induces a tight frame with frame constant $1$ and 
            \begin{align*}
                |V_gg(z)| \leq C (1+|z|)^{-s},\qquad V_gg(\lambda) \neq 0\quad \text{for } \lambda \in \Lambda \cap B(0, r + 3l_M)
            \end{align*}
            for some $C > 0$ and $s > 2d+1$ where $l_M$ is the diameter of the fundamental domain of $\Lambda$. Then there exists constants $C_1, C_2$ only dependent on $g$, the lattice $\Lambda$ and the radius $r$ such that
		\begin{align*}
                C_1 \# \partial^{r_\Lambda}_\Lambda B(0, R) \leq \Vert \rho_{B(0,R)} - \chi_{B(0,R)} \Vert_{\ell^1(\Lambda)} \leq C_2 \# \partial^{r_\Lambda}_\Lambda B(0,R)
		\end{align*}
            for all $R > 0$, where $r_\Lambda = r + l_M$ and $l_M$ is the diameter of the fundamental domain of $\Lambda$.
	\end{reptheorem}
        We will prove the theorem by showing that the size of the plunge region grows at least as fast as the boundary of the balls. The theorem on which the following lemma is based uses a different measure of the size of the boundary so we have to use a geometric argument to show that it is equivalent to the $\partial^r_\Lambda \Omega$ measure that we use.
	\begin{lemma}\label{lemma:plunge_size}
		Let $g \in L^2(\R^d)$ be such that
		\begin{align*}
		|V_gg(z)| \leq C(1+|z|)^{-s}
		\end{align*}
            for some $C > 0$ and $s > 2d+1$. Assume further that there is a positive number $r$ such that
		\begin{align*}
		      V_gg(\lambda) \neq 0\qquad \text{ for } \lambda \in B(0, r+2l_M) \cap \Lambda.
		\end{align*}
		Then there exists a $\delta > 0$ and a constant $c$ such that
		\begin{align*}
		c \# \partial^{r}_\Lambda B(0,R) \leq \#\big\{ k : \delta < \lambda_k^{B(0,R)} < 1-\delta \big\}
		\end{align*}
            for all $R > l_M$.
	\end{lemma}
        The proof of the lemma essentially boils down to translating \cite[Theorem 5.5.3]{Feichtinger2003} to our case where we measure the size of the boundary by $\# \partial^r_{\Lambda} \Omega$. For the readers convenience, we repeat a version of it here.	
        \begin{theorem}[{\cite[Theorem 5.5.3]{Feichtinger2003}}]
            Let $g \in L^2(\R^d)$ be such that
		\begin{align*}
		|V_gg(z)| \leq C(1+|z|)^{-s}
		\end{align*}
		for some $C > 0$ and $s > 2d+1$. Assume further that there is a positive number $r$ such that the fundamental domain of $\Lambda$ is contained in $B(0, r)$ and
		\begin{align*}
		      V_gg(\lambda) \neq 0\qquad \text{ for } \lambda \in B(0, r) \cap \Lambda.
		\end{align*}
            If $\mathcal{O}$ be a family of finite subsets $\Omega_\Lambda \subset \Lambda$ satisfying
            \begin{align}\label{eq:Sk_ineq}
                \# S_{\Omega_\Lambda}^k \leq C \# \partial^r \Omega_\Lambda,
            \end{align}
            for some $C > 0$, where
            \begin{align*}
                \partial^r \Omega_\Lambda = \big\{ \lambda \in \Omega_\Lambda : B(\lambda, r) \cap \Omega_\Lambda^c \neq \emptyset \big\} \cup \big\{ \lambda \in \Omega_\Lambda^c : B(\lambda, r) \cap \Omega_\Lambda \neq \emptyset \big\}
            \end{align*}
            and
		\begin{align*}
		S_{\Omega_\Lambda}^k = \big\{ \lambda \in \Omega_\Lambda : k \leq d(\lambda, \Omega_\Lambda^c) < k+1 \big\},
		\end{align*}
            then for any $\delta > 0$ sufficiently close to $0$, there is a positive constant $c$ such that for all $\Omega_\Lambda \in \mathcal{O}$,
            \begin{align*}
                c \#\partial^r \Omega_\Lambda \leq \# \{ k : \delta < \lambda_k^\Omega < 1-\delta \}.
            \end{align*}
        \end{theorem}
        \begin{proof}[Proof of Lemma \ref{lemma:plunge_size}]    
            Ultimately our goal is to apply the above theorem to the collection of balls $(B(0,R))_{R > 0}$ although due to the difference in boundary measures we cannot apply it directly. We relate them by showing that
            \begin{align*}
                \# \partial^{r}_\Lambda B(0,R) \leq \# \partial^{r+2l_M} (B(0,R) \cap \Lambda).
            \end{align*}
            Indeed, if $\lambda \in \partial^{r}_\Lambda B(0,R)$ we know that $\lambda$ is within distance $r$ to a point $\omega \in \partial \Omega$. If $\lambda \in B(0,R)$, let $\omega_c = \frac{\omega}{|\omega|} (R + l_M)$ and if $\lambda \in B(0,R)^c$, let $\omega_c = \frac{\omega}{|\omega|} (R - l_M)$. Then since $R > l_M$ by assumption, $\omega_c$ is a point on the other side of $\partial B(0,R)$. Now $B(\omega_c, l_M)$ is fully contained in the opposite side of $\partial B(0,R)$ again by $R > l_M$ and since $B(\omega_c, l_M)$ has radius $l_M$, it must contain a lattice point $\lambda_e$. By the triangle inequality,
            \begin{align*}
                d(\lambda, \lambda_e) \leq d(\lambda, \omega) + d(\omega, \omega_c) + d(\omega_c, \lambda_e) \leq r + 2l_M
            \end{align*}
            and so we conclude that $\lambda \in \partial^{r + 2l_M} B(0, R)$.

            From \cite[Remark 5.5.4 (ii)]{Feichtinger2003} we know that \eqref{eq:Sk_ineq} is satisfied for dilated balls so we can apply the theorem with the radius $r + 2l_M$ since the theorem assumptions have been included in the assumptions of the lemma.
	\end{proof}

        \begin{remark}
            In \cite{Feichtinger2003} it is stated that the proofs of the results of \cite[Section 5.5]{Feichtinger2003}, including \cite[Theorem 5.5.3]{Feichtinger2003}, ``will appear elsewhere'' but this has not yet appeared.
        \end{remark}
 
	With the above lemma in place, we are ready to proceed with the proof of the theorem.
	\begin{proof}[Proof of Theorem \ref{theorem:sharpness}]
            Using Theorem \ref{theorem:main_l1} with $A=B = 1$ we get the upper bound with an additional term $\frac{1}{\Vert g \Vert_{L^2}^2}$. Since the theorem is trivially true if $R = 0$, we can assume that $\# \partial_\Lambda^r B(0,R) \geq 1$ and this means we can absorb the additional term in the $C_2$ constant and thus it only remains to prove the lower bound.
            
            From the proof of Lemma \ref{lemma:l1_eigenval_rewrite} we have that
		\begin{align*}
                \Vert \rho_\Omega - \chi_\Omega \Vert_{\ell^1(\Lambda)} &= \frac{1}{\Vert g \Vert_{L^2}^2} \left( A_\Omega - \sum_{k=1}^{A_\Omega} \lambda_k^\Omega + \# (\Omega \cap \Lambda)\Vert g \Vert_{L^2}^2 - \sum_{k=1}^{A_\Omega} \lambda_k^\Omega \right)\\
                &= \frac{1}{\Vert g \Vert_{L^2}^2} \left( A_\Omega - \sum_{k=1}^{A_\Omega} \lambda_k^\Omega + \sum_{k=1}^\infty \lambda_k^\Omega - \sum_{k=1}^{A_\Omega} \lambda_k^\Omega \right)
		\end{align*}
            where we used that the frame is tight with frame constant $1$ which yields equality in \eqref{eq:non_tight_ineq} for the first equality and Lemma \ref{lemma:sum_of_eigs} for the last equality. This can be bounded using purely eigenvalues as
		\begin{align}\nonumber
                A_\Omega - \sum_{k=1}^{A_\Omega} \lambda_k^\Omega + \sum_{k=1}^\infty \lambda_k^\Omega - \sum_{k=1}^{A_\Omega} \lambda_k^\Omega &= \sum_{k=1}^{A_\Omega} (1 - \lambda_k^\Omega) + \sum_{k=A_\Omega+1}^\infty \lambda_k^\Omega\\\label{eq:eigen_sum_l1_bound}
		      &\geq\sum_{k=1}^\infty \lambda_k^\Omega(1 - \lambda_k^\Omega).
		\end{align}
		Meanwhile, by applying Lemma \ref{lemma:plunge_size} with radius $r_\Lambda$ we know that there exists a $\delta > 0$ and $c > 0$ such that
		\begin{align*}
		c \# \partial^{r_\Lambda}_\Lambda B(0,R) \leq \#\big\{ k : \delta < \lambda_k^{B(0, R)} < 1-\delta \big\}
		\end{align*}
		for all $R > l_M$. Letting $P \subset \N$ denote the indices $k$ such that $\delta < \lambda_k^{B(0, R)} < 1-\delta$ for this $\delta$, we have that 
		\begin{align*}
		\sum_{k=1}^\infty \lambda_k^{B(0, R)}\big(1 - \lambda_k^{B(0, R)}\big) &\geq \sum_{k \in P} \lambda_k^{B(0, R)}\big(1 - \lambda_k^{B(0, R)}\big)\\
		&\geq \delta^2 |P| \geq \delta^2 c \# \partial^{r_\Lambda}_\Lambda B(0,R).
		\end{align*}
            Plugging this into \eqref{eq:eigen_sum_l1_bound}, we can relate it to $\Vert \rho_\Omega - \chi_\Omega \Vert_{\ell^1(\Lambda)}$ with $\Omega = B(0,R)$ to get
		\begin{align*}
                \Vert \rho_{B(0,R)} - \chi_{B(0,R)} \Vert_{\ell^1(\Lambda)} \geq \frac{1}{\Vert g \Vert_{L^2}^2}\delta^2 |P| \geq \delta^2 c \# \partial^{r_\Lambda}_\Lambda B(0,R)
		\end{align*}
		which finishes the proof in the $R > l_M$ case.

            The $R \leq l_M$ case can be treated manually by first noting that $\# \partial_\Lambda^{r_\Lambda} B(0,R)$ can be bounded from above by $\#\big( \Lambda \cap B(0, R + r_\lambda) \big)$. Meanwhile $\Vert \rho_{B(0,R)} - \chi_{B(0,R)} \Vert_{\ell^1(\Lambda)}$ can be bounded from below as follows. From the beginning of this proof, we know that we can crudely bound
            \begin{align*}
                \Vert \rho_\Omega - \chi_\Omega \Vert_{\ell^1(\Lambda)} = \sum_{k=1}^{A_\Omega} (1 - \lambda_k^\Omega) + \sum_{k=A_\Omega+1}^\infty \lambda_k^\Omega \geq 1-\lambda_1^{B(0,R)} \geq 1-\lambda_1^{B(0, l_M)} > 0
            \end{align*}
            where we in the last step used the general eigenvalue bound from \eqref{eq:eigenvalue_bounds}. Then by choosing $C_1 = \frac{1-\lambda_1^{B(0,l_M)}}{\#( \Lambda \cap B(0, R + r_\lambda) )}$, we get that
            \begin{align*}
                C_1 \# \partial_\Lambda^{r_\Lambda} B(0, R) &\leq C_1 \#\big( \Lambda \cap B(0, R + r_\lambda) \big)\\
                &= 1-\lambda_1^{B(0,l_M)} \leq \Vert \rho_{B(0,R)} - \chi_{B(0,R)} \Vert_{\ell^1(\Lambda)}
            \end{align*}
            finishing the proof.
	\end{proof}
	Note that if we choose $g$ to be the standard Gaussian, the conditions of Lemma \ref{lemma:plunge_size} are fulfilled so the above theorem is not vacuous.
 
	\section{Hyperuniformity of Weyl-Heisenberg ensembles on lattices}\label{sec:wh_ensemble}
	As we saw in Section \ref{sec:prelim_rkhs}, the image space $V_g(L^2) \subset \ell^2(\Lambda)$ is a reproducing kernel Hilbert space. The corresponding RKHS for the continuous STFT was used to induce a determinantal point process on $\R^{2d}$ in \cite{Abreu2017, Abreu2019, Abreu2023}, the Weyl-Heisenberg ensemble, and by the same procedure we can induce a determinantal point process $\mathcal{X}_g$ on $\Lambda$ with $k$-point intensities
	\begin{align*}
	\rho_k(\lambda_1, \dots, \lambda_k) = \det\big( [K_g(\lambda_i, \lambda_j)]_{1 \leq i,j \leq k} \big)
	\end{align*}
	where $K_g$ is the reproducing kernel $K_g(\lambda, \lambda') = \langle \pi(\lambda')g, \pi(\lambda)g \rangle$ from \eqref{eq:lattice_gabor_rk} \cite{Feichtinger2014, GAF_book}.
	
	A point process $\mathcal{X}$ is said to be \emph{hyperuniform} if the variance of the number of points in a large ball grows slower than the volume \cite{Torquato2003, Torquato2018}. Letting $\mathcal{X}(\Omega)$ denote the points in a set $\Omega$, this means that
	\begin{align}\label{eq:hyperuniform_def_growth}
	   \mathbb{V}\big[ \mathcal{X}(B(0,R)) \big] = o(R^{2d}).
	\end{align}
	In particular, when the growth is on the order of $R^{2d-1}$, the point process is said to be \emph{class I} hyperuniform.
	
	A standard formula (see e.g. \cite[Proposition 1.E.1]{Gautier2020}) specialized to the case of a lattice $\Lambda$ states that for a determinantal point process with correlation kernel $K$,
	\begin{align}\label{eq:variance_calc}
	\mathbb{V}\left[ \sum_{x \in \mathcal{X}} f(x) \right] = \sum_{\lambda \in \Lambda} f(\lambda)^2 K(\lambda, \lambda) - \sum_{\lambda \in \Lambda} \sum_{\lambda' \in \Lambda} f(\lambda) f(\lambda') K(\lambda, \lambda') K(\lambda', \lambda).
	\end{align}
	We will use this formula to show that our determinantal point process $\mathcal{X}_g$ on $\Lambda$, induced by a tight Gabor frame, is class I hyperuniform.
	\begin{reptheorem}{theorem:hyperuniformity}
            Let $g \in M^*_\Lambda(\R^d)$ and $\Lambda$ be such that $(g, \Lambda)$ induces a tight frame, then the determinantal point process $\mathcal{X}_g$ on $\Lambda$ with correlation kernel $K_g(\lambda, \lambda') = \langle \pi(\lambda') g, \pi(\lambda)g \rangle$ is hyperuniform.
	\end{reptheorem}
	\begin{proof}
            We can assume that the frame constant is $1$ without loss of generality by multiplying $g$ by a constant if necessary. With $K_g(\lambda, \lambda') = \langle \pi(\lambda')g, \pi(\lambda)g \rangle$, we have that $K_g(\lambda, \lambda) = \Vert g \Vert_{L^2}^2$ and $K_g(\lambda, \lambda') = \overline{K_g(\lambda', \lambda)}$. Hence, with $\phi(\lambda-\lambda') = |K(\lambda, \lambda')|^2$, we can write \eqref{eq:variance_calc} as
		\begin{align*}
                \mathbb{V}\big[ \mathcal{X}_g(\Omega) \big] = \sum_{\lambda \in \Lambda} \Vert g \Vert_{L^2}^2 \chi_\Omega(\lambda) - \sum_{\lambda \in \Lambda} \sum_{\lambda' \in \Lambda} \chi_\Omega(\lambda) \chi_\Omega(\lambda') \phi(\lambda - \lambda').
		\end{align*}
            From here Lemma \ref{lemma:sum_to_boundary_estimate} applies and bounds the variance by $C \Vert g \Vert_{L^2}^2 (\Vert g \Vert_{M^*_\Lambda}^2 + 1) \# \partial^{r_\Lambda}_\Lambda \Omega$. We claim that this implies the desired growth behavior \eqref{eq:hyperuniform_def_growth} when specialized to the case $\Omega = B(0,R)$. Indeed, we can write
            \begin{align*}
                \# \partial^{r_\Lambda}_\Lambda B(0, R) = \# \big( \Lambda \cap (B(0, R+r_\Lambda) \setminus B(0, R-r_\Lambda))\big)
            \end{align*}
            whenever $R > r_\Lambda$ and this cardinality can be bounded in the same way as was done in Proposition \ref{prop:only_d1_works}. Let $A$ denote the annulus $B(0, R+r_\Lambda) \setminus B(0, R-r_\Lambda)$ and cover $A$ by a collection $Q$ of hypercubes of side length $l_m/\sqrt{2d}$ where $l_m$ is the smallest distance between two points in $\Lambda$. Then each hypercube contains at most one lattice point so $\#(\Lambda \cap A) \leq \# Q$. Meanwhile, the hypercubes can be contained in a larger annulus so by comparing volumes we find
            \begin{align*}
                \bigcup_{q \in Q} q \subset A + B(0, l_m) &\implies \# Q \left(\frac{l_m}{\sqrt{2d}}\right)^{2d} \leq \frac{\pi^d}{d!} \big( (R + r_\Lambda + l_m)^{2d} - (R - r_\Lambda - l_m)^{2d} \big)\\
                &\implies \#(\Lambda \cap A) \leq \# Q = O(R^{2d-1})
            \end{align*}
            which is what we wished to show.
	\end{proof}
        
        \subsection*{Acknowledgments}
        The author appreciates the reviewers’ detailed and thoughtful comments, which led to significant improvements of the manuscript.
        
        This project was partially supported by the Project Pure Mathematics in Norway, funded by Trond Mohn Foundation and Tromsø Research Foundation.
	
	\printbibliography
	
\end{document}